\documentclass[12pt]{amsart}
\usepackage[foot]{amsaddr}
\usepackage{amssymb,amsmath,amsthm,comment}
\usepackage{hyperref}
\usepackage{float}

\usepackage{setspace}
\newcommand{\shrinkmargins}[1]{
  \addtolength{\textheight}{#1\topmargin}
  \addtolength{\textheight}{#1\topmargin}
  \addtolength{\textwidth}{#1\oddsidemargin}
  \addtolength{\textwidth}{#1\evensidemargin}
  \addtolength{\topmargin}{-#1\topmargin}
  \addtolength{\oddsidemargin}{-#1\oddsidemargin}
  \addtolength{\evensidemargin}{-#1\evensidemargin}
  }
\shrinkmargins{0.5}
\newtheorem{theorem}{Theorem}

\newtheorem{corollary}[theorem]{Corollary}
\newtheorem{thext}{Theorem}

\newtheorem{proposition}[theorem]{Proposition}

\theoremstyle{definition}

\theoremstyle{remark}
\newtheorem{remark}[theorem]{Remark}
\newtheorem{example}[theorem]{Example}

\def\func#1{\mathop{\rm #1}}%

\begin{document}
\title{Sequences with Inequalities}
\author{Bernhard Heim$^{1,2}$}
\address{$^{1}$Department of Mathematics and Computer Science\\Division of Mathematics\\University of Cologne\\ Weyertal 86--90 \\ 50931 Cologne \\Germany}
\address{$^{2}$Lehrstuhl A f\"{u}r Mathematik, RWTH Aachen University, 52056 Aachen, Germany}
\email{bheim@uni-koeln.de}
\author{Markus Neuhauser$^{3,4,\ast }$}
\address{$^{3}$Kutaisi International University, 5/7, Youth Avenue,  Kutaisi, 4600 Georgia}
\email{markus.neuhauser@kiu.edu.ge}
\address{$^{4}$Lehrstuhl f\"{u}r Geometrie und
Analysis, RWTH Aachen University, 52056 Aachen, Germany}
\email{neuhauser@mathga.rwth-aachen.de}
\thanks{$^{\ast }$Corresponding author}
\begin{abstract}
We consider infinite sequences of positive numbers.
The connection between
log-concavity and the Bessenrodt--Ono inequality had been in the focus of several papers.
This has applications in the white noise distribution theory
and combinatorics. We improve a recent result by
Benfield and Roy and show that for the sequence of partition numbers $\{p(n)\}$
Nicolas' log-concavity result implies the result by Bessenrodt and Ono towards $p(n) \, p(m) > p(n+m)$.
We provide several examples. Benfield and Roy gave a
conjecture related to $\ell $-ary partition numbers.
We prove part of this conjecture.
\end{abstract}
\maketitle
\newpage
\section{Introduction}
Newton (\cite{HLP52}, page 104) discovered that the coefficients $\{\alpha(n)\}_{n=0}^d$ of polynomials with
positive coefficients are log-concave if all the roots are real
and therefore are unimodal:
\begin{equation}\label{logconcave}
\alpha(n)^2 \geq \alpha(n-1) \,\,\, \alpha(n+1), \quad 1 \leq n\leq d-1.
\end{equation}

Nevertheless, non-real rooted polynomials as the chromatic polynomials related to the four-color conjecture 
have this propery \cite{Hu12} and play an important role in
algebraic geometry and Hodge theory to combinatorics \cite{Ka22}.
Infinite sequences seem to be not accessable by these methods,
although,  for example $p(n)$ the number of partitions \cite{On04}  of $n$,
have
also turned out to be log-concave \cite{Ni78,DP15} for all even positive 
$n$ and all odd numbers $n >25$ proven by the circle method by
Hardy--Ramanujan and an exact formula
by Rademacher.
It has been discovered that (\ref{logconcave}) is
also related to the hyperbolicity of the Jensen
polynomials associated with
the reciprocal of the Dedekind eta function
(the generating function of the $p(n)$) of degree two, in the context of studying the Riemann hypothesis \cite{GORZ19}.
Recently, another property had been discovered
by Bessenrodt and Ono \cite{BO16},
which is a mixture of additive and multiplicative behaviour:
\begin{equation*}
p(n) \, p(m) > p(n+m) , \,\,\, n,m \geq 2 \text{ and } n+m>9.
\end{equation*}
This inequality appears at a first glance as a very special property.
However, additional sequences were found to satisfy
this Bessenrodt--Ono inequality
and similiar inequalities
appear as conditions for a white noise distribution theory
\cite{AKK00, AKK01}.
Initiated by the work
by Asai, Kubo, and Kuo there is evidence that
under certain conditions log-concavity does imply the 
Bessenrodt--Ono inequality, see Corollary~\ref{cor:akk}.

\begin{thext}[Asai, Kubo, Kuo]
\label{thm:2}Let $\{\alpha(n)\}_{n=0}^{\infty}$
be a sequence of positive numbers normalized with $\alpha(0)=1$.
Then 
\begin{equation*}
\big\{\alpha(n)^2 \leq
\alpha(n-1) \alpha(n+1), \, n \geq 1 \big\}
\Longrightarrow
\big\{\alpha(n) \alpha(m) \leq \alpha(n+m), \, n,m \geq 0\big\}.
\end{equation*}
\end{thext}

We also refer to (\cite{GMU24}, Theorem 4.4).
Since many important sequences, as $p(n)$, are not log-concave for all $n$, Benfield and Roy
\cite{BR24a} came up with
an interesting result, motivated by the sequence of partition numbers.

\begin{thext}[Benfield, Roy]
Let $\{\alpha(n)\}_{n=0}^{\infty}$ be a
sequence of
positive real numbers. Let a natural number $N$ exist, such that for all $n>N$ the sequence is log-concave: $\alpha(n)^2 \geq \alpha(n-1) \alpha(n+1)$. 
Let there be a $k\geq 0$ such that the single condition 
\begin{equation}\label{condition:BR}
\left( \alpha(N+k
) \right)^{\frac{1}{N+k
}} \geq 
\left( \alpha(N+k+1
) \right)^{\frac{1}{N+k+1
}}
\end{equation}
be valid. Then $ \alpha(m+n) \leq \alpha(m) \, \alpha(n)$ for all $m,n \geq
N+k$.
\end{thext}

Benfield and Roy applied their result to the sequence
$\{p(n)\}_{n=0}^{\infty}$. Since for $N= 26
$ and $k=0$ the condition (\ref{condition:BR})
is fulfilled, they obtain from the log-concavity property that
the Bessenrodt--Ono inequality
holds for $m,n >25$.

Unfortunately, this did
not cover the complete result
by Bessenrodt--Ono. There are still infinitely many pairs $(m,n)$
left to be checked.
For example the Bessenrodt--Ono inequality
holds true for all pairs
$(2,n)$, where $n >25$.

In this paper, motivated by the result
by Benfield and Roy,
we give a new criterion,
again, providing a sufficient condition (see
Example~\ref{hinreichend}),
but strong enough to
imply the complete result
by Bessenrodt and Ono (see Corollary~\ref{thm:bo}):

\begin{theorem}
\label{brrichtig}
Let a sequence $\{\alpha(n)\}_{n=0}^{\infty}$ with $\alpha(n) >0$ be given.
Suppose there is an $n_{0} \geq 1$ such for all $n \geq n_0$
the sequence is log-concave, $ \alpha(n)^2 \geq \alpha(n-1) \, \alpha(n+1)$,
for all $n \geq n_{0}$ and satisfies
\begin{equation}
\alpha(n_{0})^{1/n_{0}}>
\frac{\alpha(n_{0})}{\alpha(n_{0}-1)}.
\label{eq:bedingung}
\end{equation}
Let 
$$A
=\left\{ 1
\leq a\leq n_{0}-2\, : \, 
\alpha(a)^{1/a} >
\frac{\alpha(
{n_{0}})}{
\alpha({n_{0}-1})}
\right\} $$ and let further $m,n \geq 1$ satisfy 
\begin{equation*}
\big(m,n \geq n_0-1\big) \, \vee \, \big(\, m \in A, n \geq n_0-1\big) \, \vee \, 
\big( m \geq n_0-1, n \in A\big).
\end{equation*}
Then the Bessenrodt--Ono inequality $\alpha(m) \, \alpha(n) > \alpha(m+n)$ is satisfied.
\end{theorem}

\begin{remark}
{}

\begin{enumerate}
\item The index $n_0$ 
is not unique and may be chosen suitably.

\item  After Theorem~\ref{brrichtig}
was published on arXiv, see \cite{HN24},
an improved version \cite{BR24b} of 
Benfield and Roy's result appeared also covered by our
result.
\end{enumerate}
\end{remark}

We illustrate the
previous theorem with two examples.

\begin{example}[see (\cite{BR24b}, 1.1, ex.\ 4)]
Let $\{\alpha(n)\}$ be given by
$\alpha \left( 0\right) =1/2$ and
$\alpha \left( n\right) =2^{\sqrt{n}}$. Then the sequence is log-concave for
all $n\geq 1$. Let $n_{0}=2$. Then $\alpha(2)^{1/2}> \alpha(2)/\alpha(1)$.
Therefore, condition (\ref{eq:bedingung})
is fulfilled and we obtain the strict
Bessenrodt--Ono inequality as stated in
Theorem~\ref{brrichtig} for
$n,m\geq n_{0}-1=1$. In this case, the set $A$ is empty and plays no role.
\end{example}

\begin{example}[see \cite{BR24b}, 1.1, ex.\ 5)]
Define the sequence
$\alpha \left( n\right) =\left( 1.24\right) ^{n}+0.000001$
for $0\leq n\leq 24$ and
$\alpha \left( n\right) =p\left( n\right) $ for
$n\geq 25$. This sequence is log-concave for $n\geq 26$.
Let $n_0=27
$.
Then condition (\ref{eq:bedingung}) is fulfilled and 
$A=\left\{ 2,3,\ldots ,25\right\} $.
This leads to
the strict Bessenrodt--Ono
inequality for all
$n,m\geq 2$ with $n\geq 26$ \emph{or\/}
$m\geq n_{0}-1=26$.

Note that
taking $n_{0}=26$ we obtain the strict
Bessenrodt--Ono inequality {only}
for all $n,m\geq 25$.
\end{example}
In the case of partition numbers $p(n)$ Theorem \ref{brrichtig}
leads to the following interesting application.

\begin{corollary}
\label{thm:bo}Let the sequence $\{p(n)\}$
of partition numbers be given. Then the log-concavity result
by Nicolas \cite{Ni78} implies
the Bessenrodt--Ono inequality result
by Bessenrodt and Ono \cite{BO16}.
\end{corollary}
\begin{proof}
Let $\alpha(n)=p(n)$ and $n_0=26$. Then (\ref{eq:bedingung}) is satisfied.
Moreover, $$A = \{ a\in \mathbb{N}_{0}:2 \leq a <
25 \}.$$
Therefore, only finitely many
pairs given by
$
\{ \left( m,n\right) \in \mathbb{N}_{0}^{2}\, : \, 2 \leq m,n <25
\} $ remain to be checked by a computer calculation.
\end{proof}
The sequence $\{c^n\}$ for $c>0$ is log-concave but
does not satisfy the strict Bessenrodt--Ono inequality. Nevertheless, a minor modification of
Theorem \ref{brrichtig} can be applied. 
Let
\begin{equation*}
\alpha(n_{0})^{1/n_{0}} \geq
\frac{\alpha(n_{0})}{\alpha(n_{0}-1)}
\end{equation*}
and
$$A :=\left\{ 1
\leq a\leq n_{0}-2
\, : \, 
\alpha(a)^{1/a} \geq
\frac{\alpha(
{n_{0}})}{
\alpha({n_{0}-1})}
\right\} .$$ 
If $m,n \geq 1$ satisfy
\begin{equation*}
\big(m,n \geq n_{0}-1\big) \, \vee \, \big(\, m \in A, n \geq n_{0}-1 \big) \, \vee \, 
\big( m \geq n_{0}-1, n \in A\big),
\end{equation*}
then
\begin{equation}
\alpha(m) \, \alpha(n) \geq
\alpha(m+n).
\label{eq:boschwach}
\end{equation}

In the last section we invest into a recent conjecture
by Benfield and Roy related to $\ell $-ary partition
numbers \cite{BR24b}.

\section{Proof of Theorem \ref{brrichtig}}
For our purposes we quickly derive the following from the proof
of Theorem~\ref{thm:2} in \cite{AKK00}.

\begin{corollary}[to the proof of Theorem~\ref{thm:2}]
\label{cor:akk}Let $\left\{ \alpha
\left( {n}\right) \right\}_{n=0}^{\infty}$ be a sequence of positive
numbers with $\alpha
\left( {0}\right) >1$.
If $\left\{ \alpha
\left( {n}\right) \right\} _{n=0}^{\infty}$
is log-concave,
then
\begin{equation} \label{eq:bo}
\alpha
\left( {n}\right) \alpha
\left( {m}\right) >\alpha 
\left( {n+m}\right) , \quad \forall n, m\geq 0
.
\end{equation}
\end{corollary}

\begin{proof}
Let $\alpha \left( n\right) $ be the members
of a sequence as in
Theorem~\ref{thm:2}.
For such a sequence the authors of \cite{AKK00} arrive
for $n\geq 0$ and $m\geq 1$ in the first line on page~84 at the
inequality
$\alpha \left( n\right) \alpha \left( m\right) \leq \alpha \left( 0\right) \alpha \left( n+m\right) $.
If we now assume $\alpha \left( 0\right) <1$ (instead of
$\alpha \left( 0\right) =1$) this shows
\begin{equation}
\alpha \left( n\right) \alpha \left( m\right) <\alpha \left( n+m\right)
.
\label{eq:striktboinv}
\end{equation}

Now we consider a sequence consisting of $\alpha
\left( {n}\right) $ with assumptions as stated in the
corollary. Since our sequence is log-concave if we put
$\beta
\left( n\right) =1/\alpha
\left( {n}\right) $ this sequence will be log-convex and we obtain
from (\ref{eq:striktboinv}) by inversion $\alpha
\left( {n}\right) \alpha
\left( {m}\right) >\alpha
\left( {n+m}\right) $.
\end{proof}

With this auxiliary result we can now finish the proof of
Theorem~\ref{brrichtig}.

\begin{proof}[Proof of Theorem \ref{brrichtig}]
Firstly, let $a,b\geq n_{0}-1$. Then we define
\[
\beta
\left( {n}\right) =\left\{
\begin{array}{ll}
\left( \frac{\alpha
\left( {n_{0}-1}\right) }{\alpha
\left( {n_{0}}\right) }\right) ^{n_{0}-n}\alpha
\left( {n_{0}}\right) , & n<n_{0}, \\
\alpha
\left( {n}\right) , & n\geq n_{0}.
\end{array}
\right.
\]
Then
$\beta
\left( {0}\right) =\left( \frac{\alpha
\left( {n_{0}-1}\right) }{\alpha
\left( {n_{0}}\right) }\right) ^{n_{0}}\alpha
\left( {n_{0}}\right) >
1$
by (\ref{eq:bedingung}). The complete sequence consisting of the $\beta
\left( {n}\right) $
is now log-concave and we can apply the corollary
of the
result by Asai, Kubo, and Kuo
\cite{AKK00}.
From Corollary~\ref{cor:akk} with $\beta
\left( {0}\right) >1$ we obtain $\beta
\left( {a}\right) \beta
\left( {b}\right) >\beta
\left( {a+b}\right) $
for all $a,b\geq 0$.
Since $\beta
\left( {n}\right) =\alpha
\left( {n}\right) $ for all $n\geq n_{0}-1$ the
Bessenrodt--Ono inequality (\ref{eq:bo}) is also fulfilled for all
$a,b\geq n_{0}-1$.

We have $\left( \alpha
\left( {n}\right) \right) ^{2}\geq \alpha
\left( {n-1}\right) \alpha
\left( {n+1}\right) $ for $n\geq n_{0}$. Since all $\alpha
\left( {n}\right) $ are
positive we have $$\frac{\alpha
\left( {n}\right) }{\alpha
\left( {n-1}\right) }\geq \frac{\alpha
\left( {n+1}\right) }{\alpha
\left( {n}\right) }.$$

Now we
assume $a\in A$ and $b\geq n_{0}-1$.
Then
$$\frac{\alpha
\left( {a+b}\right) }{\alpha
\left( {b
}\right) }=\prod _{k=b
+1}^{a+b}\frac{\alpha
\left( {k}\right) }{\alpha
\left( {k-1}\right) }\leq \left( \frac{\alpha
\left( {n_{0}
}\right) }{\alpha
\left( {n_{0}-1}\right) }\right) ^{a
}.$$
Therefore, by our assumption $\frac{\alpha
\left( {a+b}\right) }{\alpha
\left( {b
}\right) }<
\alpha
\left( {a}\right) $ which implies $\alpha
\left( {a+b} \right) <
\alpha
\left( {a}\right) \alpha
\left( {b}\right) $.
\end{proof}

\section{Applications and variants}

The following example shows that our conditions
$$ \left( \alpha \left( n_{0}\right) \right) ^{1/n_{0}}>\frac{\alpha \left( n_{0}\right) }{\alpha \left( n_{0}-1\right) }
\text{ and }
\left( \alpha \left( a\right) \right) ^{1/a}>\frac{\alpha \left( n_{0}\right) }{\alpha \left( n_{0}-1\right) }$$
are sufficient but not necessary to obtain the Bessenrodt--Ono inequality.

\begin{example}
\label{hinreichend}We consider the sequence given in Table~\ref{folgebsp3}
with
$\alpha
\left( {0}\right) $ and $\alpha
\left( {1}\right) $ positive.
\begin{table}[H]
\[
\begin{array}{rcccc}
\hline
n=&2&3&4&n\geq 5\\ \hline \hline
\alpha \left( n\right) =&7&5&15&\frac{30}{\left( n-4\right) !}\\ \hline
\end{array}
\]
\caption{\label{folgebsp3}Definition of the sequence from
Example~\ref{hinreichend}}
\end{table}
The sequence consisting of the $\alpha
\left( {n}\right) $ is log-concave for $n\geq n_{0}=4$, see
Table~\ref{logkonkavbsp3}.
\begin{table}[H]
\[
\begin{array}{rcccc}
\hline
n=&3&4&5&n\geq 6\\ \hline \hline
\frac{\left( \alpha \left( n\right) \right) ^{2}}{\alpha \left( n-1\right) \alpha \left( n+1\right) }=
&\frac{25}{105}<1
&\frac{225}{150}>1&\frac{900}{225}>1&\frac{n-3}{n-4}>1\\ \hline
\end{array}
\]
\caption{\label{logkonkavbsp3}Log-concavity and -convexity in
Example~\ref{hinreichend}}
\end{table}
Furthermore, we obtain the Bessenrodt--Ono inequality for all instances
$a,b\geq 2$, see Table~\ref{BObsp3}.
\begin{table}[H]
\[
\begin{array}{rrlrlrlrl}
\hline
a\backslash b&&2&&3&&4&b&\geq 5\\ \hline \hline
2&49&>15&35&>30&105&>15
&\frac{210}{\left( b-4\right) !}&>\frac{30}{\left( b-2\right) !}\\
3&35&>30&25&>15&75&>5
&\frac{150}{\left( b-4\right) !}&>\frac{30}{\left( b-1\right) !}\\
4&105&>15&75&>5&225&>\frac{5}{4}
&\frac{450}{\left( b-4\right) !}&>\frac{30}{b!}\\
a\geq 5&\frac{210}{\left( a-4\right) !}&>\frac{30}{\left( a-2\right) !}
&\frac{150}{\left( a-4\right) !}&>\frac{30}{\left( a-1\right) !}
&\frac{450}{\left( a-4\right) !}&>\frac{30}{a!}
&\frac{900}{\left( a-4\right) !\left( b-4\right) !}
&>\frac{30}{\left( a+b-4\right) !}\\ \hline
\end{array}
\]
\caption{\label{BObsp3}Instances of the Bessenrodt--Ono
inequality
$\alpha \left( a\right) \alpha \left( b\right) 
>\alpha \left( a+b\right) $
for Example~\ref{hinreichend}}
\end{table}
But for $n_{0}=4$ we obtain
$\left( \alpha \left( n_{0}
\right) \right) ^{1/n_{0}
}=15^{1/4}<3=\frac{\alpha \left( n_{0}
\right) }{\alpha \left( n_{0}-1
\right) }$
and $\left( \alpha
\left( {2}\right) \right) ^{1/2}=\sqrt{7
}<3
=\frac{\alpha
\left( {n_{0}}\right) }{\alpha
\left( {n_{0}-1}\right) }$.
\end{example}

On the other hand we have the following in the opposite direction.

\begin{proposition}
Let $\alpha
\left( {n}\right) $ be a sequence of positive numbers
that
is log-concave for all
$n\geq 1$ and which satisfies
the Bessenrodt--Ono
inequality (\ref{eq:boschwach})
\begin{equation*}
\alpha
\left( n
\right) \alpha
\left( m
\right) \geq \alpha
\left( n
+m
\right)
\end{equation*}
for all $n
,m
\geq 0
$. Then
$\left( \alpha
\left( {n}\right) \right) ^{1/n}\geq \frac{\alpha
\left( {n}\right) }{\alpha
\left( {n-1}\right) }$ for all $n\geq 1$.
\end{proposition}

\begin{proof}
We have
$\alpha
\left( {n}\right) =\alpha
\left( {0}\right) \prod _{k=1}^{n}\frac{\alpha
\left( {k}\right) }{\alpha
\left( {k-1}\right) }\geq \left( \frac{\alpha
\left( {n}\right) }{\alpha
\left( {n-1}\right) }\right) ^{n}$
as $\frac{\alpha
\left( {k}\right) }{\alpha
\left( {k-1}\right) }\geq \frac{\alpha
\left( {k+1}\right) }{\alpha
\left( {k}\right) }$.
\end{proof}

\begin{example}
Now we consider the plane partition numbers
$\func{pp}\left( n\right) $. By \cite{OPR22} this sequence is
log-concave for $n\geq n_{0}=12$. We can check that
$\left( \func{pp}\left( 12\right) \right) ^{1/12}>
\frac{\func{pp}\left( 12\right) }{\func{pp}
\left( 11\right) }
$.
We obtain
\[
A=\left\{ 1\leq a\leq 10:\left( \func{pp}\left( a\right) \right) ^{1/a}>\frac{\func{pp}\left( 12\right) }{\func{pp}\left( 11\right) }\right\} =\left\{ a\in \mathbb{N}_{0}:2\leq a\leq 10\right\}
.
\]
Computationally checking the cases for $2\leq
m,n\leq 10$ then
yields our previous result with Tr\"{o}ger, see
\cite{HNT}.
\end{example}

\begin{example}
Let
$\sum _{n=0}^{\infty }\bar{p}\left( n\right) q^{n}=\prod _{n=1}^{\infty }\frac{1+q^{n}}{1-q^{n}}$
define the overpartitions. By the result by
Engel \cite{En17} it is strictly log-concave
$\left( \bar{p}\left( n\right) \right) ^{2}>\bar{p}\left( n-1\right) \bar{p}\left( n+1\right) $
for $n\geq 3$. Nevertheless it is already
log-concave (\ref{logconcave}) for all
$n\geq 1$. Let $n_{0}=4$. Then
$\left( \bar{p}\left( n_{0}\right) \right) ^{1/n_{0}}>\frac{\bar{p}\left( 4\right) }{\bar{p}\left( 3\right) }$
and
\[
A=\left\{ 1\leq a\leq 2:\left( \bar{p}\left( a\right) \right) ^{1/a}>\frac{\bar{p}\left( 4\right) }{\bar{p}\left( 3\right) }\right\} =\left\{ 1,2\right\}
.
\]
By our result we obtain that the
Bessenrodt--Ono inequality (\ref{eq:bo}) is fulfilled for all
$n\geq 3$ or $m\geq 3$. By
computationally checking the cases
$1\leq n,m\leq 2$ we obtain the result of
\cite{LZ21}: the overpartitions fulfill the
Bessenrodt--Ono inequality (\ref{eq:bo}) for all $n,m\geq 1$ except
for
$\left( n,m\right) \in \left\{ \left( 1,1\right) ,\left( 1,2\right) ,\left( 2,1\right) \right\} $
where we have equality.
Equality also holds when $n=0$ or $m=0$.
\end{example}

\begin{example}
Let
$\sum _{n=0}^{\infty }p_{k
}\left( n\right) q^{n}=\frac{\prod _{n=1}^{\infty }\left( 1-q^{kn}\right) }{\prod _{n=1}^{\infty }\left( 1-q^{
n}\right)
}$ define the $k$-regular partitions.

\begin{enumerate}
\item  Let $k=2$.
By a result by Dong and Ji \cite{DJ24} this is log-concave for $n\geq n_{0}=33$.
We obtain $\left( p_{2}\left( 33\right) \right) ^{1/33}>
\frac{p_{2}\left( 33\right)
}{p_{2}\left( 32\right)
}$ and
\[
A=\left\{ 1\leq a\leq 31:\left( p_{2}\left( a\right) \right) ^{1/a}>
\frac{p_{2}\left( 33\right) }{p_{2}\left( 32\right) }\right\} =\left\{ a\in \mathbb{N}_{0}:3\leq a\leq 31\right\}
.
\]
By computationally checking the cases $3\leq n,m\leq 31$ we obtain the result by
Beckwith and Bessenrodt \cite{BB} regarding the Bessenrodt--Ono inequality.

\item  Now let $k=3$. Then again by
\cite{DJ24}
this is log-concave for
$n\geq n_{0}=58$. We obtain
$\left( p_{3}\left( 58
\right) \right) ^{1/58
}>\frac{p_{3}\left( 58\right) }{p_{3}\left( 57\right) }$
and
\[
A=\left\{ 1\leq a\leq 56:\left( p_{3}\left( a\right) \right) ^{1/a}>\frac{p_{3}\left( 58\right) }{p_{3}\left( 57\right) }\right\} =\left\{ a\in \mathbb{N}_{0}:2\leq a\leq 56\right\}
.
\]
Computationally checking the cases $2\leq n,m\leq 56$ we obtain the result by
Beckwith and Bessenrodt \cite{BB} regarding the Bessenrodt--Ono
inequality.
\end{enumerate}
\end{example}

\begin{proposition}
Let $\left\{ \alpha \left( n\right) \right\}
_{n=0}^{\infty }$
be a sequence of positive numbers
such that
\[
\left( \alpha \left( n\right) \right) ^{\frac{1}{n}}
>\left( \alpha \left( n+1\right) \right) ^{\frac{1}{n+1}}\]
for all $n\geq n_{0}\geq 1$. Then
$\alpha \left( n\right) \alpha \left( m\right) >\alpha \left( n+m\right) $
for all $n,m\geq n_{0}$.
\end{proposition}

\begin{proof}
We obtain via a direct computation
\begin{eqnarray*}
\alpha \left( n\right) \alpha \left( m\right)
&=&\left( \left( \alpha \left( n\right) \right) ^{\frac{1}{n}}\right) ^{n}
\left( \left( \alpha \left( m\right) \right) ^{\frac{1}{m}}\right) ^{m}\\
&>&\left( \left( \alpha \left( n+m\right) \right) ^{\frac{1}{n+m}}\right) ^{n}
\left( \left( \alpha \left( n+m\right) \right) ^{\frac{1}{n+m}}\right) ^{m}
=\alpha \left( n+m\right)
.
\end{eqnarray*}
\end{proof}

\section{On a conjecture by Benfield and Roy}

Benfield and Roy made a conjecture
(\cite{BR24b}, Conjecture~5.1) on the
strict Bessenrodt--Ono inequality of
$\ell $-ary partition numbers
$b^{\ell }\left( n\right) $ of $n$.
It is well known (\cite{GMU24}, Example~4.7)
that the $\ell $-ary
sequences are not log-concave. We first invest in the
case $\ell =5$. Obviously
$b^{5}\left( n\right) b^{5}\left( m\right)
\leq b^{5}\left( n+m\right) $
if $m$ or $n\in \left\{ 1,2,3,4\right\} $. This is
similar to case $m$ or $n\in \left\{ 1\right\} $ in
\cite{BO16}. For the
other cases they conjectured that $7\leq m\leq 9$ and
$6\leq n\leq 9$.
But since
$b^{5}\left( 6\right) =2
=b^{5}\left( 7
\right) $
and $b^{5}\left( 13
\right) =3
$
this has to be modified. The modified conjecture is
recorded in Table~\ref{b5}.

\begin{table}[H]
\[
\begin{array}{ll}
\hline
b^{5}\left( n\right)
b^{5}\left( m\right)
=b^{5}\left( n+m\right)
&b^{5}\left( n\right)
b^{5}\left( m\right)
<b^{5}\left( n+m\right)
\\ \hline \hline
m=1,n\equiv 0,1,2,3\pmod{5}
&m=1,n\equiv 4\pmod{5}\\
m=2,n\equiv 0,1,2\pmod{5}
&m=2,n\equiv 3,4\pmod{5}\\
m=3,n\equiv 0,1\pmod{5}
&m=3,n\equiv 2,3,4\pmod{5}\\
m=4,n\equiv 0\pmod{5}
&m=4,n\equiv 1,2,3,4\pmod{5}\\
m, n\leq 9,n+m\geq 15&\\ \hline
\end{array}
\]
\caption{\label{b5}Conjectured equality and failure of
$b^{5}\left( n\right)
b^{5}\left( m\right)
\geq b^{5}\left( n+m\right) $ for $n\geq m\geq 1$}
\end{table}


\begin{theorem}
The conjecture holds true for all pairs $m,n\leq 7500$.
\end{theorem}

This has been verified with PARI/GP.

The same applies to the $\ell $-ary partition
numbers $b^{\ell }\left( n\right) $
(Table~\ref{br}).

\begin{table}[H]
\begin{center}
\begin{tabular}{ll}
\hline
$b^{\ell }\left( n\right)
b^{\ell }\left( m\right)
=b^{\ell }\left( n+m\right) $
&$b^{\ell }\left( n\right)
b^{\ell }\left( m\right)
<b^{\ell }\left( n+m\right) $
\\ \hline \hline
$1\leq m\leq \ell -1$
with&$1\leq m\leq \ell -1$ with\\
$n\equiv 0,1,\ldots ,\ell -m-1\pmod{\ell }$
&$n\equiv \ell -m, \ldots ,\ell -1\pmod{\ell }$
\\
$m, n\leq 2\ell -1,m+n\geq 3\ell $&\\ \hline
\end{tabular}
\end{center}
\caption{\label{br}Conjectured equality and failure of
$b^{\ell }\left( n\right)
b^{\ell }\left( m\right)
\geq b^{\ell }\left( n+m\right) $ for $n\geq m\geq 1$
and $\ell\geq 5$}
\end{table}

In the following we
prove that all these
cases of equality and failure in Table~\ref{br} occur.
For $1\leq s\leq \ell $ and
$\left( s-1\right) \ell \leq n\leq s\ell -1$ holds
$
b^{\ell }\left( n\right) =s
$.
For $0\leq r\leq \ell -1$ also
$b^{\ell }\left( s\ell +r\right) =b^{\ell }\left( s\ell \right) $ holds true
for all $s\geq 0$ since the smallest part in a partition
different from $1$ is $\ell $.
For the small values $1\leq m\leq \ell -1$ we obtain therefore
$b^{\ell }\left( n\right) b^{\ell }\left( m\right) =b^{\ell }\left( n\right)
=b^{\ell }\left( n+m\right) $
if $n\equiv 0,1,\ldots ,\ell -m-1\pmod{\ell }$. On the other hand
$b^{\ell }\left( \ell s-1\right) <b^{\ell }\left( \ell s\right) $. Therefore,
$b^{\ell }\left( n\right) b^{\ell }\left( m\right) =b^{\ell }\left( n\right)
<b^{\ell }\left( n+m\right) $ if
$n\equiv \ell -m,\ell -m+1,\ldots ,\ell -1\pmod{\ell }$.
Furthermore for $\ell \leq n,m\leq 2\ell -1$ we obtain
$2\ell \leq n+m\leq 4\ell -2$ and therefore
$b^{\ell }\left( n\right) b^{\ell }\left( m\right) =4
>3=b^{\ell }\left( n+m\right) $
if $n+m\leq 3\ell -1$ and
$b^{\ell }\left( n\right) b^{\ell }\left( m\right) =4
=b^{\ell }\left( n+m\right) $
if $n+m\geq 3\ell $.

Now consider generally
$\left( s-1\right) \ell \leq m\leq s\ell -1$ and
$\left( t-1\right) \ell \leq n\leq t\ell -1$ with
$n+m\leq \ell ^{2}-1$. Then
$b^{\ell }\left( n\right) b^{\ell }\left( m\right) =ts$ and
$b^{\ell }\left( n+m\right) =t+s$ if $n+m\geq \left( t+s-1\right) \ell $ and
$b^{\ell }\left( n+m\right) =t+s-1$ if $n+m\leq \left( t+s-1\right) \ell -1$.
Now we obtain $ts=t+s$ if and only if
$\left( t-1\right) \left( s-1\right) =ts-t-s+1=1$ i.~e.\ $t=2=s$ and
$ts=t+s-1$ if and only if $\left( t-1\right) \left( s-1\right) =ts-t-s+1=0$
i.~e.\ $t=1$ or $s=1$. Similarly $ts<t+s$ if and only if
$\left( t-1\right) \left( s-1\right) <1$ i.~e.\ $t=1$ or $s=1$ and
$ts<t+s-1$ if and only if $\left( t-1\right) \left( s-1\right) <0$ i.~e.\
never.

With PARI/GP we verified the conjecture for $\ell =6$
and $m,n\leq 3500$ and for $\ell =7$ and $m,n\leq 1200$.

\begin{example}
Consider the following
sequence consisting of $\alpha \left( n\right) =2^{n}+1$ for $n$ odd and
$\alpha \left( n\right) =2^{n}+2$ for $n$ even. Then for $n$ odd we obtain
\begin{eqnarray*}
\left( \alpha \left( n\right) \right) ^{2}
-\alpha \left( n-1\right) \alpha \left( n+1\right)
&=&\left( 2^{n}+1\right) ^{2}-\left( 2^{n-1}+2\right) \left( 2^{n+1}+2\right) \\
&=&2^{2n}+2^{n+1}+1-2^{2n}-2^{n+2}-2^{n}-4\\
&=&-3\cdot 2^{n}-3<0
\end{eqnarray*}
which fulfills the log-convexity condition and for $n$ even we obtain
\begin{eqnarray*}
\left( \alpha \left( n\right) \right) ^{2}
-\alpha \left( n-1\right) \alpha \left( n+1\right)
&=&\left( 2^{n}+2\right) ^{2}-\left( 2^{n-1}+1\right) \left( 2^{n+1}+1\right) \\
&=&2^{2n}+2^{n+2}+4-2^{2n}-2^{n-1}-2^{n+1}-1\\
&=&3\cdot 2^{n-1}+3>0
\end{eqnarray*}
which fulfills the log-concavity condition. Therefore, the sequence is
neither log-concave nor log-convex. Consider now for
arbitrary $n$
$\left( \alpha \left( n\right) \right) ^{n+1}\geq
\left( 2^{n}+1\right) ^{n+1}
=\left( 2^{n}+1\right) ^{n}\left( 2^{n}+1\right)
>\left( 2^{n}+1\right) ^{n}2^{n}=\left( 2^{n+1}+2\right) ^{n}
\geq \left( \alpha \left( n+1\right) \right) ^{n}$.
For all $n$ therefore
$\left( \alpha \left( n\right) \right) ^{\frac{1}{n}}
>\left( \alpha \left( n+1
\right) \right) ^{\frac{1}{n+1}}$
holds true.
We can
check
that
\[
\alpha \left( n\right) \alpha \left( m\right)
\geq \left( 2^{n}+1\right) \left( 2^{m}+1\right) =2^{n+m}+2^{n}+2^{m}+1
>2^{n+m}+2\geq \alpha \left( n+m\right)
\]
for all $n,m$.

To summarize we have a sequence that is neither
log-concave nor log-convex, satisfies the Bessenrodt--Ono
inequality, and for which
$\left( \alpha \left( n\right) \right) ^{1/n}
>\left( \alpha \left( n+1\right) \right)
^{1/\left( n+1\right) }$.
\end{example}

\section*{Declarations and statements}

\subsection*{Funding}

This research received no external funding.

\subsection*{Competing
interests}

The authors declare that they have no competing interests.

\subsection*{Compliance with
ethical
standards}

The authors declare that they comply with ethical
standards

\subsection*{Data
availability
statement}

No datasets were generated or analysed during the current study.

\subsection*{Author's contribution statement}

Both authors contributed in equal parts.


\end{document}